\documentclass[12pt]{amsart}
\usepackage{amssymb,latexsym, amsmath, amscd, array, graphicx}

\swapnumbers \numberwithin{equation}{section}

\theoremstyle{plain}

\newtheorem{thm}{Theorem}[section]

\newtheorem{conjec}[thm]{Conjecture}
\newtheorem{prop}[thm]{Proposition}
\newtheorem{cor}[thm]{Corollary}
\newtheorem{prob}[thm]{Problem}

\theoremstyle{definition}

\newtheorem{question}[thm]{Question}

\DeclareMathOperator{\cat}{{\mbox{\rm cat$_{\rm LS}$}}}

\def\C{{\mathbb C}}
\def\Z{{\mathbb Z}}

\def\R{{\mathbb R}}

\def\1{\hbox{\rm\rlap {1}\hskip.03in{\rom I}}}
\def\Bbbone{{\rm1\mathchoice{\kern-0.25em}{\kern-0.25em}
{\kern-0.2em}{\kern-0.2em}I}}

\long\def\forget#1\forgotten{} %

\newcommand\ver[1]{\marginpar{\tiny Changed in Ver \VER}}

\date{\today}
\begin{document}

\title[On some problems]{On some problems related to the Hilbert-Smith conjecture}

\author[A.~Dranishnikov]{Alexander  Dranishnikov} %
\thanks{Supported by NSF, grant DMS-1304627}

\address{Alexander N. Dranishnikov, Department of Mathematics, University
of Florida, 358 Little Hall, Gainesville, FL 32611-8105, USA}
\email{dranish@math.ufl.edu}


\maketitle

\section{Introduction}

The 5th Hilbert's problem~\cite{H} has an unsettled extension known as the Hilbert-Smith conjecture~\cite{Sm},\cite{Wi}:
\begin{conjec}[Hilbert-Smith Conjecture]
If a compact group $G$ acts effectively (freely) on a connected manifold, then $G$ is a Lie group.
\end{conjec}
It is known that the Hilbert-Smith conjecture is equivalent to the question whether the group of $p$-adic integers
$A_p$ can act effectively (freely) on a manifold~\cite{Sm2}. 

The Conjecture is proven for $n$-dimensional manifolds
with $n\le 2$~\cite{MZ2} and $n=3$~\cite{P}. For arbitrary $n$ the Hilbert-Smith Conjecture is proven for smooth actions on a smooth manifold~\cite{MZ2}, for Lipschitz actions on  Riemannian manifolds~\cite{RSc},
for H\" older actions~\cite{Mal}, and for quasi-conformal actions~\cite{Mar}.

A quite deep but not yet successful line of research on the Hilbert-Smith Conjecture is known as the
orbit space method. For an effective action of a compact group $G$ on a space $X$ the formula for the dimension of the orbit space $$\dim X/G=\dim X-\dim G$$ seemed quite natural.
Since $A_p$ is homeomorphic to the Cantor set and hence, $\dim A_p=0$, one would expect the dimension of the orbit space $M/A_p$ of an effective action of an
$n$-manifold to be equal $n$. Contrary to that P.A. Smith found in 1940~\cite{Sm}
that this dimension is not $n$. Later
C.T. Yang proved~\cite{Y1} that the cohomological dimension $\dim_{\Z}M/A_p$ of the orbit space of an effective $p$-adic action on an $n$-manifold $M$ equals $n+2$. Therefore, by Alexandroff's theorem~\cite{A},\cite{Wa} about the coincidence of
cohomological and covering dimensions in the case when the latter is finite  either
$\dim M/A_p=n+2$ or $\dim M/A_p=\infty$.
The other surprising  properties of the orbit spaces of  a hypothetical $p$-adic action on a manifold can be found in~\cite{BRW},
\cite{R},\cite{RW},\cite{Wi},\cite{Y2}. Still there is a remote hope that those bizarre properties of the orbit space $M/A_p$
could lead to a contradiction and prove the Hilbert-Smith Conjecture.

In this paper we consider the Hilbert-Smith Conjecture under the assumption that the dimension of the orbit space is finite.
We consider only free actions.

\begin{conjec}[Weak Hilbert-Smith conjecture]
 If a compact group $G$ acts freely on a manifold $M$  and dimension of the orbit space $M/G$ is finite, then $G$ is a Lie group.
 \end{conjec}
 
We reduce the weak Hilbert-Smith conjecture to two problems
 which we call the {\em Essential Lens Sequence} problem (Problem~\ref{Pr2}) and the {\em  Injectivity Conjecture}
 (Conjecture~\ref{C1}).
 The reduction is based on the idea of Williams to use the infinite product in the K-theory
 of some special version of a classifying space $BA_p$ for  the group $A_p$ ~\cite{Wi}.
We should warn the reader that in the definition of classifying spaces for $A_p$ in~\cite{Wi} the order of direct and inverse limit must be exchanged. 

The Essential Lens Sequence problem is a quest for compact 'classifying space' (we call it here a {\em rough classifying space}) which has
an infinite product in K-theory  and hence infinite dimensional. A finite dimensional compact rough classifying space for $A_p$ was constructed by Floyd~\cite{Wi}. Existence of such an infinite dimensional rough classifying space together with Borel's construction would bring 
examples of cell-like maps which are Hurewicz fibrations and yet have nontrivial cokernels in K-theory.
This looks surprising but it does not contradict to any known facts about cell-like maps. Thus, it is known that cell-like maps of manifolds
can have nontrivial kernels in homology K-theory~\cite{DFW}.

The Injectivity Conjecture is a technical statement about Hurewicz fibrations with the fiber a closed  manifold $M$
over a fixed base $B$. It states that for a fixed nonzero generalized cohomology class $\alpha\in h^*(B\times M)$ for a sufficiently
close approximation $f:E\to B\times M$ of the trivial fibration $\pi:B\times M\to B$ by a Hurewicz fibration $p:E\to B$ with the fiber $M$ the map $f$ takes $\alpha$ to a nonzero element $f^*(\alpha)$. We apply the Injectivity Conjecture when $h^*$ is the reduced K-theory. We note that for general spaces $B$ the concept of Hurewicz fibration is a peculiar one,
since $B$ does not necessarily support interesting homotopies. There is a seemingly weaker notion of a 
completely regular map which does not appeal to homotopy in its definition. It is still unknown (the Hurewicz Fibration Problem)
whether every completely regular map is a Hurewicz fibration.
We conclude the paper by reductions of the Injectivity Conjecture first
to  the Hurewicz Fibration Problem and then to the ANRness problem of the classifying space for completely regular maps with a given manifold fiber.

\begin{thm}
A positive solutions to both the Essential Lens Sequence Problem and the Hurewicz  Fibration Problem imply the
Weak Hilbert-Smith Conjecture for closed aspherical manifolds.
\end{thm}

\begin{thm}
If the Essential Lens Sequence Problem has positive solution and the classifying space for completely regular maps with a given compact $Q$-manifold is an absolute neighborhood extensor for  compact metric spaces,
then the
Weak Hilbert-Smith Conjecture holds true.

\end{thm}
 
Our approach to the Hilbert-Smith Conjecture does not exclude a possibility of a $p$-adic action on a manifold with an infinite dimensional orbit space. On the other hand still there is no known examples of   $p$-adic actions on a finite dimensional
compact space with an infinite dimensional orbit space. Though there are examples of such actions of Cantor groups~
 \cite{DW},\cite{Le}.

\section{Preliminaries}

\subsection{$p$-Adic actions}
By $\Z_m=\Z/m\Z$ we denote the cyclic group of oder $m$. 
Let $p$ be a prime number.
The $p$-adic integers 
is a topological group defined as the inverse limit of the sequence
$$
\Z_p\leftarrow \Z_{p^2} \leftarrow\Z_{p^3}\leftarrow\dots 
$$
where every bonding map $\Z_{p^{k+1}}\to\Z_{p^k}$ is the  mod $p^k$ reduction.
Note that every closed subgroup of $A_p$ has the form $p^kA_p$. If the group $A_p$ acts on a 
compact metric space $X$, then $X$ can be presented as the limit space
of the inverse sequence
$$
\begin{CD}
(*)\ \ \ \ \ Y_0 @<q^1_0<< Y_1 @<q^2_1<< Y_2 @<q^3_2<< Y_3 @<<<\dots
\end{CD}
$$
with  $Y_0=X/A_p$ and each space $Y_k$ equals to the orbit space $X/p^kA_p$ of the action of the subgroup $p^kA_p$, all the bonding maps $q^{k+1}_k$ are the projection to the orbit space of a $\Z_p$-action, and every composition $$q^{k+i}_k=q^{k+1}_k\circ q^{k+2}_{k+1}\dots\circ q^{k+i}_{k+i-1}:Y_{k+i}\to Y_k$$ is the projection onto the orbit space of an action of the quotient group $\Z_{p^i}=p^kA_p/p^{k+i}A_p$.

\begin{prop}
Suppose that the group of $p$-Adic integers $A_p$ acts freely on a connected and locally connected
compact metric space $X$ with the orbit space $Y$.
Then $X$ and $Y$ can be presented as the inverse limit of sequences of simplicial complexes such that there is a commutative diagram
$$
\begin{CD}
K_0 @<\phi^1_0<< K_1 @<\phi^2_1<< K_2 @<\phi^3_2<<  K_3@<<< \cdots\ \ @. X\\
@Vp_0VV @Vp_1VV @Vp_2VV @Vp_3VV @. @Vq_0VV\\
L_0 @<\psi^1_0<< L_1 @<\psi^2_1<< L_2 @<\psi^3_2<< L_3 @<<<  \cdots\ \  @. Y\\
\end{CD}
$$
where each $p_k$ is the projection onto the orbit space of a free action of $\Z_{p^k}$,
each bonding map $\phi^k_{k-1}$ is $\Z_{p^k}$-equivariant.
\end{prop}
\begin{proof}
We take $Y=Y_0$ from $(*)$ and construct the inverse sequence $\{L_i,\psi^i_{i-1}\}$ using nerves
of a sequence of finite open covers $\{\mathcal U_i\}$ with $mesh\ \mathcal U_i\to 0$.
Since $Y$ is locally connected, we may assume that all sets in each $\mathcal U_i$ are connected.
Thus, $L_i=Nerve(\mathcal U_i)$. Let $\psi_i:Y\to L_i$ denote the projection to the nerve.
We recall that it is defined by means of a partition of unity subordinated by $\mathcal U_i$

We may assume that  each $U\in\mathcal U_i$ 
admits a section of $q^i_0$. Thus, the preimage
$(q_0^i)^{-1}(U)$ is a disjoint union of $p^i$ copies of $U$. These copies of $U$ define a finite open cover $\mathcal V_i$ of $Y_i$. Let $K_i=Nerve(\mathcal V_i)$. Note that there is a simplicial map $q_i:K_i\to L_i$ which is the projection onto the orbit space of $\Z_{p^i}$-action. Moreover, there is a map $\psi_i':Y_i\to K_i$ which defines a pull-back diagram
$$
\begin{CD}
K_i @<\psi'_i<< Y_i\\
@Vp_iVV @Vq^i_0VV\\
L_i @<\psi_i<< Y.
\end{CD}
$$
It is an easy exercise to show that the multi-valued upper semi-continuous map $F:K_i\to K_{i-1}$
defined by the formula $F(x)=\psi'_{i-1}(\psi'_i)^{-1}$ is in fact single-valued. Thus $F$ defines a continuous map
$\phi^i_{i-1}$. We show that the square diagram in our sequence commutes $p_{i-1}\phi^i_{i-1}=\psi^i_{i-1}p_i$:
$$
p_{i-1}\phi^i_{i-1}=p_{i-1}F(x)=p_{i-1}\psi'_{i-1}((\psi'_i)^{-1}(x))
=\psi_{i-1}q^i_0((\psi'_i)^{-1}(x))$$
$$=\psi_{i-1}(\psi_i^{-1}(p_i(x)))=\psi^i_{i-1}p_i(x).$$

Clearly, $\lim_{\leftarrow}\{K_i,\phi^i_{i-1}\}=X$.
\end{proof}

\subsection{Borel construction}
Let a group $G$ act on spaces $X$ and $E$ with the projections onto the orbit spaces $q_X:X\to X/G$ and $q_E:E\to E/G$. Let 
$q_{X\times E}:X\times E\to X\times_G X=(X\times E)/G$ denote the projection to the orbit space of the diagonal action of $G$ on $X\times E$. Then there is a commutative diagram
called the {\em Borel construction}~\cite{AB}:
$$
\begin{CD}
X @<pr_X<< X\times E @>pr_E>> E\\
@Vq_XVV @ Vq_{X\times E}VV @Vq_EVV\\
X/G @<p_E<< X\times_GE @>p_X>> E/G.\\
\end{CD}
$$
If $G$ is compact, the actions are free, and $q_E$ is locally trivial bundle, then so is $p_X$.
In particular, this holds true for free actions of compact Lie groups. Since the projection of the  limit
space of an inverse sequence whose bonding maps are locally trivial fibrations onto the first space 
of the sequence is a Hurewicz fibration, using approximation of a compact group by  compact Lie groups, we obtain that in the case of free $G$-actions all projections in the diagram are Hurewicz fibrations. 
The fiber $p_X^{-1}(y)$ is homeomorphic to $X/I_z$ where $I_z=\{g\in G\mid g(z)=z\}$ is the isotropy group of $z\in q_E^{-1}(y)$.

\section{Essential sequences of lens spaces}

For an integer $m$ we denote the standard $(2n-1)$-dimensional lens space mod $m$ by
$$
L^n(m)=S^{2n-1}/\Z_{m}$$ 
where the $\Z_m$-action on $(2n-1)$-sphere $S^{2n-1}\subset\C^n$ is obtained from the rotation by $2\pi/m$ in every coordinate plane $\C$ in
$\C^n$.

Note that the classifying space $B\Z_{m}=K(\Z_{m},1)$ can be presented as the increasing union $\cup_nL^n(m)$.

We call a map between lens spaces $q:L^m(p^k)\to L^n(p^{\ell})$ {\em essential} if it induces an epimorphism of the fundamental groups. 
A sequence of mappings of spheres
$$
\begin{CD}
S^{k_0}  @<f^1_0<< S^{k_1} @<f^2_1<< S^{k_2} @<f^3_2<<\dots
\end{CD}
$$
is called {\em inessential} if for every $i$ there is $j$ such the $f^{i+1}_i\circ\dots\circ f^{i+j}_{i+j-1}$ is null-homotopic.

\begin{prob}[Essential Lens  Sequence Problem]\label{Pr2}
Given an odd prime $p$, does there exist an infinite sequence of lens spaces
$$
\begin{CD}
(1)\ \ \ \ \ L^{n_0}(p^{k_0}) @<q^1_0<< L^{n_1}(p^{k_1})@<q^2_1<< L^{n_2}(p^{k_2}) @<q^3_2<<\dots
\end{CD}
$$
with essential bonding maps, $k_i\to\infty$, and $n_i>p^{k_i-k_0}$, such that the sequence of covering spheres
$$
\begin{CD}
 (2)\ \ \ \ \ \ \ \ \ \ \ \ S^{2n_0-1}@<\tilde q^1_0<< S^{2n_1-1}@<\tilde q^2_1<< S^{2n_2-1}@<\tilde q^3_2<<\dots
\end{CD}
$$
is inessential\ ?
\end{prob}

It is known (see Section 4) that $d(c)$, the dimension of a lens space in an essential sequence as a function of the cardinality of its fundamental group,
can grow at most linearly. Floyd's example \cite{Wi} gives us an essential sequence of lens spaces with constant function $d(c)=3$. Using ideas from~\cite{Dr} it is possible to
construct an essential sequence with $d(c)\sim \log c$. It turns out that for our applications to the Hilbert-Smith conjecture we need $d(c)$ to be linear.
This requirement is spelled out in the condition: $n_i>p^{k_i-k_0}$.

The ELS problem can be stated for concrete values of $k_i$ and $n_i$:
$$ 
(1)\ \ \  k_i=i+1\  \ \ \ \ \ \ \  \text{and}\ \ \ \ \ \ n_i=p^{i}+1;
$$
$$ 
(2)\ \ \     k_i=i+1\  \ \ \ \ \ \ \   \text{and}\ \ \ \ \ \ n_i=p^{i}+2;
$$
$$
(3)\ \ \  k_i=2^i\ \ \ \ \ \ \ \ \  \text{and} \ \ \ \ \  n_i=p^{2^i-1}+1.
$$

Suppose that a sequence from Problem~\ref{Pr2} does exist.
Then the maps $q^{i+1}:L^{n_{i+1}}(p^{k_{i+1}})\to L^{n_i}(p^{k_i})$ define maps $\tilde q^{i+1}_i:S^{2n_{i+1}-1}\to S^{2n_i-1}$  between the universal covers. Denote by $$E=\lim_{\leftarrow}\{S^{2n_i-1},\tilde q^{i+1}_i\}$$ and by $$B=\lim_{\leftarrow}\{L^{n_i}(p^{k_i}),q^{i+1}_i\}$$ the corresponding inverse limits.
Then the following proposition is straightforward.
\begin{prop}\label{classif}
The compactum $E$ is cell-like and
there is a free $A_p$-action on $E$ with the orbit space $B$.
\end{prop}

Using Ferry's theorem~\cite{F1} we can modify the bonding maps (possibly with stabilizations) in the inverse sequence 
$\lim_{\leftarrow}\{L^{n_i}(p^{k_i}),q^{i+1}_i\}$
to $UV^0$-maps.
Then we may assume that $B$ and $E$  are path connected and locally path connected.

\subsection{K-theory of lens spaces}
 For any $r$ the actions of the groups $\Z_{p^r}\subset\Z_{p^{r+1}}\subset S^1$ on $S^{\infty}$ and $S^{2n-1}$
form a commutative diagram
$$
\begin{CD}
S^{\infty} @>>> B\Z_{p^k} @>>> B\Z_{p^{k+1}} @>>> \C P^{\infty}\\
@A\subset AA @A\subset AA @A\subset AA @A\subset AA\\
S^{2n-1} @>>> L^n(p^r) @>>> L^n(p^{r+1}) @>>> \C P^n.\\
\end{CD}
$$
The canonical line bundle $\eta$ over $\C P^{\infty}$ defines the line bundles over all spaces in the diagram. This bundle defines an
element in the K-theory of the Eilenberg-MacLane space $B\Z_{p^k}$ which  will be denoted by $\eta_k$. 

As it was computed by Atiyah (see~\cite{AS}) the K-theory of $B\Z_{p^k}$ equals the completion of the representation ring $R\Z_{p^k}$
of $\Z_{p^k}$.
We recall that, $R\Z_{p^k}=\Z[\eta]/\langle 1-\eta^{p^k}\rangle$ where $\eta$ is 
the class of the complex representation of $\Z_{p^k}$
generated by the group embedding $\eta:\Z_{p^k}\to S^1$.
Note that $\eta_k$ is obtained from $\eta$ by  passing  to the map of classifying spaces $\eta:B\Z_{p^k}\to BS^1=\C P^{\infty}$ and pulling back
the canonical complex line bundle. Thus, taking the completion we obtain
 $$K^0(B\Z_{p^k})=(R\Z_{p^k})^{\wedge}=\Z[[\eta_k]]/\langle 1-\eta_k^{p^k}\rangle$$ where $A[[x]]$ denotes the ring of formal series with the variable $x$ and coefficients in $A$.
Note that the mod $p^k$ reduction homomorphism $\Z_{p^{k+r}}\to \Z_{p^{k}}$ takes the generator $\eta_{k}$ to $\eta_{k+r}^{p^r}$.

Let $\bar\eta_k$  denote the restrictions of these classes
to the $(2n-1)$-dimensional lens space $L^n(p^k)$. The
$K$-theory of this lens space was computed in ~\cite{Ma},~\cite{Ka}: 
$$K^0(L^n(p^k))=\Z[\bar\eta_k]/\langle 1-\bar\eta_k^{p^k},(\bar\eta_k-1)^{n}\rangle.$$
We note that the ideal generated by  $\bar\eta_k-1$ in the above ring is isomorphic to the
reduced $K$-theory of $L^n(p^k)$ (see~\cite{KS}).

\begin{prop}\label{ideal}
For any positive integers $k>l$, $m<p^l$, and prime $p$ the polynomial $(x^{p^{k-l}}-1)^m$ does not belong to the ideal $\langle x^{p^k}-1,(x-1)^n\rangle$ of the polynomial ring $\Z[x]$ provided that $mp^{k-l}<n$.
\end{prop}
\begin{proof}
We change the variable $y=x-1$. Thus, we need to show that $((y+1)^{p^{k-l}}-1)^m$ does not belong to
$\langle (y+1)^{p^k}-1,y^n\rangle$.
The $mod p^{k-l}$ reduction of this problem is the question whether $y^{mp^{k-l}}$ belongs the ideal
$\langle y^{p^k},y^n\rangle$. Since $mp^{k-l}<\min\{p^k,n\}$, the answer to the question is negative
and the result follows.
\end{proof}

\begin{prop}\label{infinite}
Let $q_i:B\to L^{n_i}(p^{k_i})$ be the projection of the limit in the above inverse system. 
Suppose that $m<p^{k_i-k_0}$
Then the induced homomorphism in the reduced  K-theory $$q_i^*:\tilde K^0(L^{n_i}(p^{k_i}))\to \tilde K^0(B)$$  takes $(\bar\eta_{k_i}-1)^m$ to a nonzero element.
\end{prop}
\begin{proof}
Note that $$(q^{i+j}_i)^*((\bar\eta_{k_i}-1)^m)=(\bar\eta_{k_{i+j}}^{p^{k_{i+j}-k_i}}-1)^m.$$
We apply Proposition~\ref{ideal} with $k=k_{i+j}$, $l=k_i$, $n=n_{i+j}$ to obtain that
$(\bar\eta_{k_{i+j}}^{p^{k_{i+j}-k_i}}-1)^m\ne 0$ since $m<p^{k_i}$ and
$$mp^{k_{i+j}-k_i}<p^{k_i-k_0}p^{k_{i+j}-k_i}=  p^{k_{i+j}-k_0}<n_{i+j},$$
i.e., the conditions of Proposition~\ref{ideal} are satisfied.
\end{proof}

\begin{cor}\label{inf-prod}
The reduced K-theory cup length of $B$ is unbounded.
\end{cor}

\section{Lens Sequence Problem and the Schwarz genus}

We consider the level functions defined in ~\cite{Me}$$v_{p,k}(m):=\min\{n | \ \exists\  f:L^m(p^{k-1})\stackrel{\Z_p}\rightarrow S^{2n-1}\}$$
where $f$ is a $\Z_p$-equivariant map with respect to the standard free actions. Lower bounds for these functions were given by Vick~\cite{V}, rediscovered by Bartsch~\cite{Ba}, and formulated in this way by D. Meyer~\cite{Me}:
$$
v_{p,k}(m)\ge\lceil\frac{m-1}{p^{k-1}}\rceil+1.
$$
Also D. Meyer has computed $v_{p,2}(m)=\frac{m-2}{p}+2$ for odd $p$ and $m=2$ mod $p$~\cite{Me}.

REMARK. The existence of a sequence of lens spaces as in Problem~\ref{Pr2} does not contradict to the above estimates of the level functions. 
Indeed, for $i>j$ the composition 
$q^{j+1}_j\circ\dots\circ q^i_{i-1}$ induces a $\Z_p$-equivariant map $L^{n_i}(p^{k_i-k_j})\to S^{2n_j-1}$. Therefore, we have the inequality
$$n_{j}\ge v_{p,k_{i}-k_{j}+1}(n_{i})\ge\lceil\frac{n_{i}-1}{p^{k_i-k_j}}\rceil+1\ge \frac{p^{k_{i}-k_{0}}}{p^{k_{i}-k_j}}+1>p^{k_j-k_0}$$
which is consistent with the conditions on $k_i$ and $n_i$.

\subsection{Schwarz genus} We recall that the Schwarz genus $Sg(f)$ of a fibration $f:E\to B$ is the minimal $k$ such that $B$ can be covered by
open sets $A_1,\dots, A_k$ such that $p$ admits a section on each $A_i$ ~\cite{Sch}.

\begin{prop}\cite{Sch}\label{joint section}
$Sg(f)\le n$ if and only if $\ast^nf:\ast^n_BE\to B$ admits a section.
\end{prop}

A free action of $\Z_r$ on $S^1$ determines a free $\Z_r$-action on $S^{2m-1}=\ast^m S^1$, and a free action of $\Z_{pr}$ on $S^{2m-1}$ determines a free $\Z_r$-action on $L^m(p)$.
In the question below we consider the free $\Z_{p^k}$-actions on $L^m(p)$ and $S^1$
determined in this way.

Let  $\pi_k^m=p_{S^1}:W^m_k=L^m(p)\times_{\Z_{p^k}}S^1\to L^m(p^{k+1})$ be the  $S^1$-bundle from the Borel construction
for $\Z_{p^k}$-actions on $S^1$ and $L^m(p)$.
\begin{thm}\label{schwarz}
There is a map $q:L^m(p^{k+1})\to L^n(p^{k})$ that induces an epimorphism of the fundamental groups if and only if
$Sg(\pi_k^m)\le n$.
\end{thm}
\begin{proof}
Such map $q$ exists if and only if there is a $\Z_{p^{k}}$-equivariant map $q':L^m(p)\to S^{2n-1}$  for free actions. This equivalent to the existence of a section of the locally trivial $S^{2n-1}$-bundle
$\pi:L^m(p)\times_{\Z_{p^k}}S^{2n-1}\to L^m(p^{k+1}$ from the Borel construction. Since the sphere $S^{2n-1}=\ast_{i=1}^nS^1$ is the join product of
$n$ circles and the action comes from a free $\Z_{p^k}$-action on $S^1$, the bundle $\pi$ is the fiberwise join product of $n$ copies of
the $S^1$-bundle $\pi^m_k$. Then the result follows from Proposition~\ref{joint section}.
\end{proof}

\begin{cor}
$Sg(\pi^m_k)\ge \lceil\frac{m-1}{p}\rceil+1$.

For odd $p$, $Sg(\pi^m_1)=\frac{m-2}{p}+2$ provided $m=2$ mod $p$.
\end{cor}
\begin{proof}
Suppose that $Sg(\pi^m_k)=n$. Then there is a map $q:L^m(p^{k+1})\to L^n(p^{k})$ that induces an epimorphism of the fundamental groups. Going to $\Z_{p^k}$ covers we obtain a $\Z_{p^k}$-equivariant map
$f:L^m(p)\to S^{2n-1}$. Thus, the map $f$ is $\Z_p$-equivariant as well, $\Z_p\subset\Z_{p^k}$. 
Hence $v_{p,2}(m)\le n$. We proved the inequality $Sg(\pi^m_k)\ge v_{p,2}(m)$. The cited above results of Vick, Bartsch, and D. Meyer  imply 
$Sg(\pi^m_k)\ge\lceil\frac{m-1}{p}\rceil+1$.
\end{proof}

\begin{prob}\label{Pr3}
(1) What is the Schwarz genus of the $S^1$-bundle from the Borel construction $\pi_k^m:L^m(p)\times_{\Z_{p^k}}S^1\to L^m(p^{k+1})$?

(2) In particular, if $m=2$ mod $p$, is $Sg(\pi_k^m)=\frac{m-2}{p}+2$ for all $k$ ?
\end{prob}

\begin{cor}
A positive answer to Problem~\ref{Pr3} (2) gives a positive answer to Problem~\ref{Pr2}.
\end{cor}
\begin{proof}
We take $k_i=i$ and $n_i=3p^i-\sum_{s=1}^{i-1}p^s+1$.
Clearly, the condition $n_i>p^{k_i-k_0}$ is satisfied.
Note that $m=n_{i+1}+1$ equals 2 $mod\ p$ and
$\frac{m-2}{p}+2=3p^i-\sum_{s=1}^{i-1}p^s-1\le n_i$.
If $Sg(\pi_i^m)=\frac{m-2}{p}+2$, then by Theorem~\ref{schwarz} there is an essential map
$f^{i+1}_i:L^m(p^{i+1})\to L^{n_i}(p^i)$. We define $q^{i+1}_i:L^{n_{i+1}}(p^{i+1})\to L^{n_i}(p^i)$
to be the restriction of $f^{i+1}_i$ to $L^{n_{i+1}}(p^{i+1})\subset L^m(p^{i+1})$.
Then the corresponding sequence of spheres will be inessential.
\end{proof}

\section{A reduction of the weak Hilbert-Smith conjecture}

\subsection{Dimension, LS-category, and cup-length}
We recall that a topological space $X$ has the Lusternick-Schnirelmann category (LS-category for short) $cat(X)\le n$ if there is 
an open cover $U_0,\dots, U_n$ of $X$ by $n+1$ contractible in $X$ sets. It's known that in the case of ANR, it suffices to take closed $U_i$s 
or even arbitrary~\cite{Sr}.
We refer to ~\cite{CLOT} for general properties of  the LS-category. The basic properties are that $cat(X)$ is a homotopy invariant,
it is bounded from above by dimension $\dim X$, and from below by the length of a nonzero cup product of nonzero dimensional elements
in cohomology. It is known that the cohomology in this cup product could be generalized or even 0-dimensional if one uses
a reduced cohomology theory defined by means of  a spectrum~\cite{Sw}.
\begin{prop}
The LS-category of a finite connected complex is greater than or equal to
the cup-length for any reduced generalized cohomology theory $\tilde h^*$.
\end{prop}
\begin{proof}
This proposition can be extracted from \cite{Ru}. We give an alternative proof, since it is short.
Assume that $w=\alpha_1\smile\dots\smile\alpha_k\ne 0$ in $\tilde h^*(X)$.
Suppose that $\cat X\le k-1$. Then there is a cover $U_1,\dots, U_k$ of $X$
by contractible in $X$ subcomplexes (for some subdivision). The long exact sequence of pair $(X,U_i)$
for th
e reduced $h$-cohomology and the fact that $\tilde h^*(X)\to \tilde h^*(U_i)$ are 0-homomorphisms
imply that the homomorphisms $\tilde h^*(X;U_i)\to \tilde h^*(X)$ are isomorphisms in all dimension.
Let $\bar\alpha_i$ denote the corresponding elements in $\tilde h^*(X;U_i)$. Then
the product $\bar w=\bar\alpha_1\smile\dots\smile\bar\alpha_k\ne 0$.
On the other hand, $\bar w\in \tilde h^*(X;U_1\cup\dots\cup U_k)=\tilde h^*(X;X)=0$. We have a contradiction.
\end{proof}
\begin{cor}\label{cl-dim-cw}
The  cup-length of a finite connected complex for any generalized reduced cohomology theory
$\tilde h^*$ does not exceed the dimension of the complex.
\end{cor}

\begin{thm}\label{dim-cup}
For a finite dimensional compact metric connected space $X$, the cup-length for
any reduced generalized cohomology theory does not exceed $\dim X$.
\end{thm}
\begin{proof}
Let $\dim X=n$. Then $X$ can be presented as the inverse limit of a
sequence of $n$-dimensional polyhedra $X=\lim_{\leftarrow}L_m$. If
$$\alpha_1\smile\dots\smile\alpha_k\ne 0$$ in $\tilde h^*(X)$, then
there is $m$ such that $\alpha_i=p_m^*(\beta_i)$, $i=1,\dots,k$, and
$\beta_1\smile\dots\smile\beta_k\ne 0$ where $p_k:X\to L_k$ is the
projection in the inverse system. By Corollary~\ref{cl-dim-cw}, $k\le n$.
\end{proof}

\subsection{Injectivity conjecture}

The Chapman-Ferry $\alpha$-approximation theorem~\cite{CF2} has
severer versions. For instance Theorems 1, 2, 3, 4 in~\cite{F3} are all variations of that. One of the weakest version
states that for a fixed metric on a closed manifold $M$ for every $\delta>0$ there is $\epsilon>0$ such that every $\epsilon$-map $f:M\to M$ is  $\delta$-homotopy equivalence.
The following conjecture is a parametrized version of this version of the $\alpha$-approximation theorem with a fixed space of parameters $B$.
Since we don't assume that $B$ is nice, we replace in our conjecture the homotopy equivalence 
by a shape equivalence.

Let $F$ be a compact metric space. We say that a fibration $p:E\to B$ is a {\em fibration with isometric fibers}
$F$ if there is a metric on $E$ such that
all fibers $p^{-1}(x)$, $x\in B$, are isometric to $F$. 

\begin{conjec}[Parametrized $\alpha$-Approximation Conjecture]\label{Co1}
For every compact manifold $M$ (or $Q$-manifold) with a fixed metric on it and any connected and locally connected
compact space $B$ there is $\epsilon>0$ such that for any
Hurewicz fibration $p:E\to B$ with isometric fibers $M$ every fiberwise $\epsilon$-map $f:E\to M\times B$ 
is a shape equivalence. In particular, it induces an isomorphism of generalized cohomology groups.
\end{conjec}

The $\alpha$-Approximation Conjecture seems out of reach. For applications to the Hilbert-Smith conjecture
it suffices to prove the following.

\begin{conjec}[Injectivity Conjecture]\label{WCo1}
For every closed manifold $M$ (or $Q$-manifold) with a fixed metric on it, any connected and locally connected
compact space $B$, and any nonzero element $\alpha\in h^*(B)$  for some generalized cohomology theory $h^*$
there is $\epsilon>0$ such that for every
Hurewicz fibration $p:E\to B$ with isometric fibers $M$ and every fiberwise $\epsilon$-map $f:E\to M\times B$, 
 the image $f^*\pi^*(\alpha)\ne 0$ where $\pi:M\times B\to B$ is the projection.
\end{conjec}

We note that the $\alpha$-approximation theorem holds true for Hilbert cube manifolds ($Q$-manifolds).
Thus, it makes sense to extend the Injectivity Conjecture to $Q$-manifolds as well.
We note that the Injectivity Conjecture for $Q$-manifolds implies the Injectivity Conjecture for ordinary manifolds
via multiplication by the Hilbert cube.

Let $G$ be a compact metrizable topological group. The orbit space of a free $G$-action on a Peano continuum of a trivial shape will be called   a {\em rough classifying space} for $G$. 
\begin{cor}[Corollary of the Injectivity Conjecture]\label{C1}
Let $B$ be a  rough classifying space for $A_p$. Suppose that $A_p$ acts on a compact manifold  (or $Q$-manifold) $M$.
Then for any nonzero $\alpha\in K^0(B)$ there is $k$ such that $p_M^*(\alpha)\ne 0$ where $p_M:M\times_{A_p}E\to B$ is the projection from
the Borel construction for the action on $M$
of the subgroup $p^kA_p\cong A_p$.
\end{cor}
\begin{proof} Let $\mu$ be an invariant measure on $A_p$.
Integration on $A_p$ with respect to $\mu$ of a metric $d$ on $M$ gives
an $A_p$-invariant metric on $M$: $\rho(x,y)=\int_{A_p} d(gx,gy)d\mu$. 
We take any metric $d'$ on $E$ and consider the $\ell_1$-product metric $\rho+d'$ on $M\times E$.
This defines the quotient metric on each of the orbit spaces $E_k=M\times_{A_p}E$ for the diagonal action of $A_p$ where for the action on $M$ the group $A_p$ is identified with $p^kA_p$.
Thus, $p_M:E_k\to B$ is a fibration with isometric fibers $M$.

We claim that for large $k$ the total space $E_k=M\times_{A_p}E$ of the Borel construction of the action of $p^kA_p$ on $M$ admits a retraction onto a fiber $M$ the restriction of which to any other fiber
is an $\epsilon_k$-map with $\epsilon_k\to 0$. There are several ways to argue for this. We leave the proof to
the reader. One approach would be that the composition of the inverse to the quotient map $M\times E\to E_k$
followed by a  retraction $r:M\times E\to M$ to a fiber defines a multivalued such retraction with the diameter
$d_k$ of the images of points tending to 0. That with fact that $M$ is ANR would be sufficient to derive our claim.

Then the Injectivity Conjecture implies the required result.
\end{proof}

The following is the main result of the paper:

\begin{thm}
Suppose that  the Injectivity Conjecture hods true  and there exists an infinite sequence of lens  spaces
as in the Essential Lens Sequence problem.
Then there is no free  $A_p$-action on  a closed manifold  with a finite dimensional orbit space.
\end{thm}
\begin{proof} 
Assume that there is such a free $A_p$-action on an $n$-manifold $M$ with $\dim M/A_p<\infty$. Then by the Yang's theorem $\dim M/A_p=n+2$.
Let $B$ and $E$ be as in Proposition~\ref{classif}.
We apply the Injectivity Conjecture to the product of a desired length $$\alpha=\alpha_1\cdot\dots\cdot\alpha_{n+3}\in\tilde K^0(B)$$
defined by Corollary~\ref{inf-prod}.
Then $p_M^*(\alpha)\ne 0$ for the action of $p^kA_p$ for some $k$. 

By Proposition~\ref{classif} the projection
$p_E$ from the Borel construction for that action is a cell-like map. Since $\dim M/p^kA_p<\infty$, it is a shape equivalence~\cite{La}
and hence it induces an isomorphism in $K$-theory.
Hence for each $i$ there is $\beta_i\in K^0(M/p^kA_p)$ such that $p_E^*(\beta_i)=p_M^*(\alpha_i)$.
Hence, $\beta_1\dots\beta_{n+3}\ne 0$. 
$$
\begin{CD}
M @<pr_{\R^n}<< M\times E @>pr_E>> E\\
@Vq_{M}VV @ Vq_{M\times E}VV @Vq_EVV\\
M/p^kA_p @<p_E<< M\times_{A_p}E @>p_M>> B.\\
\end{CD}
$$

Since $\dim M/(p^kA_p)=n+2$, we obtain a contradiction with Theorem~\ref{dim-cup}.
\end{proof}

\section{Hurewiz fibration problem}

\subsection{Completely regular maps}
Dyer and Hamstrom introduced the notion of a completely regular map in \cite{DH}.
We recall that a continuous surjection $p:E\to B$ between metric spaces is called {\em completely regular} if for
each $b\in B$ and $\epsilon>0$, there exists $\delta(b,\epsilon)>0$ such that if $d_B(b,b')<\delta$, then there exists a homeomorphism 
$h:p^{-1}(b)\to p^{-1}(b')$ with $d_E(x,h(x))<\epsilon$ for all $x\in p^{-1}(b)$. 
It is known that the complete regularity of a map
does not depend on choice of metrics $d_B$ and $d_E$.

Using Michael's selection theorem~\cite{M}, Dyer and Hamstrom proved the following theorem~\cite{DH}. 
\begin{thm}\label{DH}
\label{locally trivial}
Suppose that for a compact  $F$ the space $Homeo(F)$ is locally contractible. Then every completely regular map
$p:E\to B$ with fiber $F$ and a finite-dimensional $B$ is a locally trivial fibration.
\end{thm}
The mistake in the proof presented in~\cite{DH} was corrected in~\cite{Ha}. A detailed proof can be found in~\cite{RS}.
Similar or stronger related results later were proven  in 
 \cite{K}, \cite{Se}, \cite{CF}, \cite{F2}. 

The Ferry's  $\alpha$-approximation theorem (Theorem 1 from~\cite{F3}) admits the following variation:
\begin{thm}\label{alf}
Let $M$ be a closed $n$-manifold, $n\ge 5$, with a fixed metric. Then for any $\epsilon>0$ there is $\delta>0$ such that
for every $\delta$-map $g:M\to N$ onto a closed $n$-manifold $N$ there is a homeomorphism $h:N\to M$
such that the composition $h\circ g$ is $\epsilon$-close to the identity $1_M$.
\end{thm}
Like the proof of the original Ferry's theorem, the proof of this variation is a consecutive 
applications of Theorems 2, 3, and 4 of ~\cite{F3}. Also we note that this theorem holds true for compact 
$Q$-manifolds.

We recall that a map $f:X\to Y$ of a subset $X$ of a metric space $(Y,d)$ is called an {\em $\epsilon$-move} if
$d(x,f(x))<\epsilon$ for all $x\in X$.
\begin{prop}\label{claim}
Let $M$ be a close manifold with a fixed metric. Then given $\epsilon_0>0$, there is $\delta_0>0$ such that
for every isometric embedding $M\subset X$ for any $n$-dimensional compact $Z\subset N_{\delta_0}(M)$
of the $\delta_0$-neighborhood of $M$ there is continuous $\epsilon_0$-move $r:Z\to M$.
\end{prop}
\begin{proof}
By the Lefschetz criterion of ANRs (see~\cite{Bo}, Theorem 8.1), for any $\epsilon>0$
there is $\delta>0$ such that for every
map of the vertices $f:K^{(0)}\to M$ of a $n$-dimensional simplicial complex $K$ with the condition that $d(f(v),(v'))<\delta$ for every edge $[v,v']\subset K$ there is  an extension $\bar f:K\to M$ with $diam\bar f(\Delta)<\epsilon$ for each simplex $\Delta\subset K$. 

We prove the Proposition when $Z$ is a polyhedron. For the general case (not needed in the paper)
can be obtained usin approximation of $Z$ by nerves of small open covers. 

We take 
$\epsilon<\epsilon_0/3$ to obtain $\delta<\epsilon$ from the Lefschetz criterion. Take $\delta_0<\delta/4$ and consider a triangulation of $Z$ with the mesh $\delta'$ satisfying $\delta'<\delta-2\delta_0$.
We may assume that $\delta'<\epsilon_0-\epsilon-\delta_0$. We define $r$ on the vertices of $Z$ by
sending each vertex $v$ to a nearest point of $M$. Thus, $d(v,r(v))<\delta_0$.
Then for any edge  $[v,v']$ we obtain $d(r(v),r(v'))<2\delta_0+\delta'<\delta$. Let $r:Z\to M$ be an extension from the Lefschetz criterion. Then for every $z\in Z$ we consider a simplex $\Delta\subset Z$ that contains $z$ and fix a vertex $v\in \Delta$. By the triangle inequality
we obtain $d(z,r(z))\le d(z,v)+d(v,r(v))+d(r(v),r(z))<\delta'+\delta_0+\epsilon<\epsilon_0$.
\end{proof}

\begin{thm}\label{iso-reg}
Let $\phi:X\to Y$ be a continuous map between compact metric spaces such that all point preimages
$\phi^{-1}(y)$ are isometric to a closed $n$-manifold $M$, $n\ge 5$. Then $\phi$ is a completely regular map.
\end{thm}
\begin{proof}
Let $y\in Y$ and $\epsilon>0$ be given.
Proposition~\ref{claim} implies that there is a neighborhood $U(y)$ of $ y\in Y$ such that for every $y'\in U(y)$
there is a $\delta/2$-move $r:\phi^{-1}(y')\to\phi^{-1}(y)$ where $\delta$ is taken for $\epsilon/2$
as in Theorem~\ref{alf}. Let $r':\phi^{-1}(y)\to\phi^{-1}(y')$ be a similar map back. We may assume that $r'\circ r$
is homotopic to the identity. Therefore,
we may assume that $r$ is surjective. Let $i:M\to \phi^{-1}(y')$ be an isometry.
Note that $r\circ i$ is a $\delta$-map. Hence there is a homeomorphism $h:\phi^{-1}(y)\to M$ with
$d_M(hri(z),z)<\epsilon/2$. Note that the homeomorphism $i\circ h$ is an $\epsilon$-move: $d_X(ih(x),x)=d_X(ihri(z),ri(z))\le
d_X(ihri(z),i(z))+d_X(i(z), ri(z))=d_M(hri(z),z)+d_X(i(z), ri(z))<\epsilon/2+\delta/2<\epsilon$.
Here $z\in M$ is such that $ri(z)=x$. Such $z$ exits in view of surjectivity of $r$.
\end{proof}
We note that Theorem~\ref{iso-reg} holds true for $Q$-manifolds as well.

\

\begin{question}[The Hurewicz Fibration Problem]\label{Hurewicz}
Is every completely regular map with a manifold   fiber  a Hurewicz fibration?
\end{question}
In view of Theorem~\ref{DH}, it is an open problem when the base is infinite dimensional. 
It is known that a completely regular map is a Serre fibration. We refer to~\cite{DS} for further discussion of the Fibration Problem.

\subsection{Completely regular maps in the Borel construction}

Suppose that a compact group $G$ acts freely on a metric space $E$ with the orbit space $B$. Suppose that it also acts on
a compact space $F$. We may assume that $G$ acts by isometries.  
\begin{prop}\label{Borel cr}
The projection $p_F:F\times_GE\to B$ in the Borel construction is completely regular.
\end{prop}
\begin{proof} Fix $y\in q_E^{-1}(x)$. Since  $q_E$ is open, any sequence $x_n$ converging to $x$ in $B$ admits a lift $y_n$ converging to $y$ in $E$
with respect to the orbit map $q_E:E\to B$. Then 
$$(q_{F\times E})(1_F\times c_i)((q_{F\times E})|_{F\times\{y\}})^{-1}:p_F^{-1}(x)\to p_F^{-1}(x_k)$$
 is the sequence of homeomorphisms that converges to the identity $id:p_F^{-1}(x)\to p_F^{-1}(x)$
where $c_k:y\to y_k$ is the map of one-point spaces and $q_{F\times E}:F\times E\to F\times_GE$ is the orbit map of the diagonal action.
\end{proof}

 We use the notation $Homeo(M)$ for the group of homeomorphisms of a manifold $M$ with the compact-open topology.
By $Homeo_0(M)$ we denote the subgroup of homeomorphisms isotopic to the identity. For manifolds with boundary we
use the notation $Homeo(M,\partial M)$ for the group of the homeomorphisms of $M$ which are the identity on $\partial M$.

\begin{question}\label{contraction}
Let $M$ be a manifold and $G\subset Homeo_0(M)$ be a compact  subgroup that admits a deformation $H:G\times[0,1]\to Homeo(M)$  to the unit 1.
Does  there exist  such a deformation $H$ through homomorphisms, i.e., such that $h_t=H(-,t):G\to Homeo(M)$ is a group homomorphism  for every $t$ ?
\end{question}

\begin{thm}
Suppose that Questions~\ref{Hurewicz} and~\ref{contraction} have positive answers. Then the Injectivity Conjecture holds true.
\end{thm}
\begin{proof}
Let $A_p$ act freely on a manifold $F$ (or $Q$-manifold). Let $h_t$ be a deformation of $A_p$ to the identity in $Homeo(F)$
by virtue of a family of subgroups $h_t(A_p)\subset Homeo(M)$.
We define an $A_p$-action on $F\times[0,1]$ by letting the group $h_t(A_p)$ act on $F\times\{t\}$. 
Let $B$ be a rough classifying space for $A_p$ with the universal covering $q_E:E\to B$.

The projection
$p_{F\times[0,1]}$ in the Borel construction factors through the  map $$p:(F\times[0,1])\times_{A_p}E\to B\times[0,1]$$ with the fiber $F$. We show that $p$ is completely regular.  By taking an invariant metric on $F\times[0,1]$ we may assume that $p$ has isometric fibers over $B\times t$ for every $t$.
By Theorem~\ref{iso-reg} (or by Proposition~\ref{Borel cr}) we obtain that $p$ is completely regular over
each $B\times t$. Thus, it suffices to prove that $p$ is completely regular over $b \times[0,1]$ uniformly on $b\in B$ in the following sense: The number  $\delta((b,t),\epsilon)$ from the definition of complete regularity can be chosen to be independent of $b$. For $(b,t)$ and $(b,t')$ we define a homeomorphism $h_{t,t'}$ of the fibers by fixing $e\in E$ with
$q_E(e)=b$ and identifying $p^{-1}(b,t)$ by means of the inverse of the projection to the orbit space
$q:F\times I\times E\to (F\times I)\times_{A_p}E$ with $F\times t\times e$ then  translating it to $F\times t'\times e$ and projecting by $q$ to $p^{-1}(b,t')$ . This homeomorphism does not depend on choice of $e$. The translation of $F\times t\times e$ to $F\times t'\times e$ is an $\epsilon$-move where $\epsilon$ depends on $t$ and $|t-t'|$ only
with $\epsilon\to 0$ as $t'\to t$. Thus, $h_{t,t'}$ is an $\epsilon$-move for all $b\in B$.

If Question~\ref{Hurewicz} has an affirmative answer, then $p$ is a Hurewicz fibration. Then
the identification $F\times\{1\}\times B=p^{-1}(B\times\{1\})$ extends to a fiberwise map
$F\times[0,1]\times B\to(F\times[0,1])\times_{A_p}E$ over $B\times [0,1]$.
The restriction of this fiber wise map over $B\times\{0\}$ yields a splitting of the fiberwise  map in the Injectivity Conjecture.
\end{proof}

\begin{prop}\label{defo}
Suppose that $A_p\subset Homeo(D^n,\partial D^n)$. Then $A_p$ admits a deformation $h_t:A_p\to Homeo(D^n,\partial D^n)$  to the identity
such $h_t$ is a group homomorphism  for every $t$. 
\end{prop}
\begin{proof} This follows from the Alexander's trick. 
Let $tD^n$ denote the image of $D^n$ under multiplication by $t\le1$. Extending to $D^n\setminus tD^n$
by the identity the identity homeomorphism 
of the boundary $\partial tD^n$  defines an embedding $h_t:Homeo(tD^n,\partial tD^n)\to Homeo(D^n,\partial D^n)$ of topological groups with the image of $h_t$ converging to the unit as $t\to\infty$.
By precomposing this embedding with the given embedding $A_p\to Homeo(D^n,\partial D^n)$
and the isomorphism $$(L_t)_*:Homeo(D^n,\partial D^n)\to Homeo(tD^n,\partial tD^n)$$
where $L_t:D^n\to tD^n$ is the multiplication by $t$ we obtain a desired deformation.
\end{proof}

\begin{cor}
If every completely regular map with the fiber $D^n$ is a Hurewicz fibration, then the Injectivity Conjecture holds true  
for any uniformly bounded $A_p$-action on $\R^n$.
\end{cor}

We call a $G$-action on a metric space $X$ {\em uniformly bounded} if there is an upper bound on the diameter of
orbits. Note that an action of $A_p$ on a closed aspherical manifold defines a uniformly bounded action on its universal cover.

Perhaps Edwards-Kirby's theorem would allow to extend Proposition~\ref{defo} to all manifolds.
We recall that by the Edwards-Kirby theorem~\cite{EK} (which goes back to the proof of Chernavsky's theorem on the local contractibility of $Homeo(M)$~\cite{Che})  every homotopic to the identity homeomorphism $h:M\to M$  of a closed manifold can be presented as a finite composition
$h=h_n\circ\cdots\circ h_1$ of homeomorphisms fixing the complement to a ball.

\

\section{Moduli space of topological manifolds}

Let $F$ be a compact metric space. We denote by $Emb(F)=\{\phi:F\to s\}$ the space 
of all topological embeddings of $F$ into
the pseudo interior $s=(0,1)^{\omega}$ of the Hilbert cube $Q=[0,1]^{\omega}$ with the sup metric:
$$
d(\phi_1,\phi_2)=\sup\{\|\phi_1(x)-\phi_2(x)\|\  \mid x\in F\}.$$
The pseudo interior is chosen in order to deal with tame embeddings only.

The following theorem implies that the space $Emb(F)$ is  an 
absolute neighborhood extensor for compact metrizable spaces,
$ANE(\mathcal C)$. In particular, it implies that $Emb(F)$ is
$n$-connected and locally $n$-connected for all $n$.
\begin{thm}[\cite{Ch}]\label{appr}
Let $(A,A_0)$ be a compact pair and $f:A\to s$ is a map. Then for any $\epsilon>0$
there is an $\epsilon$-close map $g:A\to s$ that agrees with $f$ on $A_0$ and is an embedding on $A\setminus A_0$. 
\end{thm}

We note that the group of homeomorphisms of $F$ taken with the compact-open topology, 
$H=Homeo(F)$, acts on $Emb(F)$ from the right by precomposing:
$\phi \to \phi\circ h$, $h\in H$.
Thus $\phi_1,\phi_2\in Emb(F)$ are in the same orbit if and only if $im\phi_1=im\phi_2$.
Note that $H$ acts on $Emb(F)$ by isometries: $d(\phi_1,\phi_2)=d(\phi_1\circ h,\phi_2\circ h)$.
We call the orbit space of such action {\em the moduli space} of $F$ and denote it by $\mathcal M(F)$.
Let $q_F:Emb(F)\to \mathcal M(F)=Emb(F)/H$ be the projection to the orbit space.
We consider the quotient metric $\rho$ on  $\mathcal M(F)$:
$$
\rho(\phi_1H,\phi_2H)=\inf\{d(\phi_1\circ h,\phi_2) \mid h\in H\}.$$
We check that $\rho$ is a metric. It is symmetric, since
$d(\phi_1\circ h,\phi_2)=d(\phi_1,\phi_2\circ h^{-1})$. If $\phi_1H\ne\phi_2H$, then $im\phi_1\ne im\phi_2$.
Then for any $h_1,h_2\in H$, $$d(\phi_1\circ h_1,\phi_2\circ h_2)\ge d_H^Q(im\phi_1,im\phi_2)>0$$ where $d_H^Q$ 
is the Hausdorff distance on the closed subsets of $Q$. Let 
for $i=1,2$ $h_i$ be such that $d(\phi_ih_i,\phi_3)-\rho (\phi_iH,\phi_3H)<\epsilon/2$. Then the triangle inequality follows
when $\epsilon\to 0$:
$$\rho(\phi_1H,\phi_2H)\le d(\phi_1h_1,\phi_2h_2)\le d(\phi_1h_1,\phi_3)+d(\phi_2h_2,\phi_3)<$$
$$<\rho(\phi_1H,\phi_3H)+
\rho(\phi_3H,\phi_2H)+\epsilon.$$

We will identify each orbit $\phi H\in\mathcal M(F)$ with the subset $\phi(F)$ of the Hilbert cube.

\begin{prop}\label{displacement}
For $F_1,F_2\in \mathcal M(F)$, $\rho(F_1,F_2)<\epsilon$ if and only if there is a homeomorphism $g:F_1\to F_2$
with the displacement $$D_g=\max\{\|g(x)-x\| \mid x\in F_1\}<\epsilon.$$
\end{prop}
\begin{proof} Let $F_i=\phi_i(F)$, $\phi_i\in Emb(F)$, $i=1,2$..

In one direction, since $d(\phi_1,\phi_2h)<\epsilon$ for some $h\in H$, we obtain that $D_g<\epsilon$ for
$g=\phi_2h\phi_1^{-1}$.

In the other direction,  $d(\phi_1,\phi_2h)<\epsilon$ for $h=\phi_2^{-1}g\phi_1$ if $D_g<\epsilon$.
\end{proof}

Let $\mathcal E(F)=\{(F',x)\in\mathcal M(F)\times s\mid x\in F'\}\subset \mathcal M(F)\times s$ and let $\nu_F:\mathcal E(F)\to\mathcal M(F)$ be the restriction of the projection onto the first factor.
We note that $\nu_F$ is completely regular.

\begin{prop}\label{universal}
For every continuous map $f:X\to \mathcal M(F)$ the pull-back $f^*(\nu_F)$ is completely regular.

For every completely regular map $p:X\to Y$  between compact metric spaces with fiber $F$ there is a continuous map $f:Y\to \mathcal M(F)$ such that
$p=f^*(\nu_F)$.
\end{prop}
\begin{proof}
Any embedding $j:X\to s$ defines a map $f:Y\to \mathcal M(F)$ by $f(y)=j(p^{-1}(y))$. Let $y_k\to y$ be a convergent sequence in $Y$. Then there is a sequence of
homeomorphisms $h_k:j(p^{-1}(y))\to j(p^{-1}(y_k))$ with $D_{h_k}\to 0$. By Proposition~\ref{displacement} $\rho(f(y),f(y_k))\to 0$. Therefore $f$ is continuous.
Clearly, $p$ is isomorphic to $f^*(\nu_F)$.
\end{proof}

\subsection{Moduli space of $Q$-manifolds}
\begin{prop}\label{Serre fibration}
Let $F$ be such that $Homeo(F)$ is locally contractible. Then

(1) $\mathcal M(F)$ is locally path connected.

(2) $q_F:Emb(F)\to\mathcal M(F)$ is a Serre fibration with the fiber $q_F^{-1}(y)\cong Homeo(F)$ for all $y\in \mathcal M(F)$.
\end{prop}
\begin{proof}
(1) Let $F_1$ and $F_2$ be elements of $\mathcal M(F)$ at a distance $\rho(F_1,F_2)<\delta$. Thus, $F_1,F_2\subset s$ and
there is a homeomorphism $h:F_1\to F_2$ with the displacement $D_h<\delta$. Let $f:F\to F_1$ be a homeomorphism.
We consider a linear homotopy $H:F\times I\to s$ between $f$ and $h\circ f$. By Theorem~\ref{appr} there is a $\delta$-approximation
$H':F\times I\to s$ of $H$ by an embedding that coincides with $H$ on $F\times\{0,1\}$.
This defines a path from $F_1$ to $F_2$
in $\mathcal M(F)$ of diameter $<2\delta$.

(2) Let $H:I^n\times I\to\mathcal M(F)$ and $h:I^n\times\{0\}\to Emb(F)$ with $q_Fh=H|_{I^n\times\{0\}}$. In view of Theorem~\ref{DH}, $H^*(\nu_F)$ is a trivial fiber bundle with fiber $F$. The map $h$ defines a trivialization of $H^*(\nu_F)$ over $I^n\times\{0\}$. The projection $I^n\times I\to I^n$ defines extension
of that trivialization to a trivialization over whole $I^n\times I$. This trivialization defines a lift $\bar H$ of $H$ that extends $h$.
\end{proof}

We use the standard notation $LC^n$ for the class of locally $n$-connected spaces.

\begin{thm}[G.S. Ungar~\cite{U}]\label{Ungar}
Let $p:E\to B$ be a Serre fibration of metric spaces, $E$ is $LC^n$, $p^{-1}(b)$ is $LC^{n-1}$ for all $b\in B$ and $B$ is
$LC^0$. Then $B$ is $LC^n$.
\end{thm}
\begin{cor}\label{homeo}
Suppose that $H=Homeo(F)$ is a locally contractible. Then $\mathcal M(F)$ is $LC^n$ for all $n$.
\end{cor}
\begin{proof}
We apply Theorem~\ref{Ungar} to the map $q_F:Emb(F)\to\mathcal M(F)$. Since $Emb(F)\in ANE$ and in view of Proposition~\ref{Serre fibration}
we obtain that $\mathcal M(F)$ is $LC^n$.
\end{proof}

We denote by $ANE(n)$ the class of absolute neighborhood extensors for $n$-dimensional compact metric spaces. Note 
Kuratowski's theorem characterizes  $ANE(n)$s as $LC^n$ spaces.

\begin{thm}\label{ANE(n)}
For any compact $Q$-manifold $F$, $\mathcal M(F)\in ANE(n)$ for all $n$.
\end{thm}
\begin{proof}
We apply Ferry's theorem~\cite{F4} which states that $Homeo(F)$ is an ANE for a Q-manifold $F$, Corollary~\ref{homeo}, and Kuratowski's characterization of $ANE(n)$.
\end{proof} 

\begin{prob}\label{anr}
Let $F$ be a compact $Q$-manifold.
Is the space $\mathcal M(F)$ an absolute neighborhood extensor for compact metric spaces? 
\end{prob}

Since $\mathcal M(F)$ is the orbit space of an action by isometries of an ANE group upon an $ANE(\mathcal C)$ space,
it would not be a big surprise that the above problem has a positive answer.
\begin{thm}
The affirmative answer to Problem~\ref{anr} implies the Injectivity Conjecture.
\end{thm}
\begin{proof} Assume that the Injectivity Conjecture failed to be true for $M$. Then there is a compact
space $B$, nonzero $\alpha\in h^*(B)$, a sequence of Hurewicz fibrations $p_k:E_k\to B$
with isometric fibers $M$, and a sequence of fiberwise $\epsilon_k$-maps $f_k:E_k\to M\times B$,
$\epsilon\to0$ such that $f_k^*\pi_B^*(\alpha)=0$ for all $k$. The latter implies that $p_k^*(\alpha)=0$
for all $k$.
Here $\pi_B$ and $\pi_M$ denote projections of the product $B\times M$ to the factors.

We define a compactification $X$ of $\coprod E_k$ by $M$ as the subspace
$$X=\coprod G_{k}\cup \{a\}\times M\subset\alpha(\coprod_kE_k)\times M$$ of the product of the one-point compactification of the union of $E_k$ and $M$ where $G_{k}\subset E_k\times M$ is the graph of the composition
$\pi_M\circ f_k$ and $a$ is the compactifying point in $\alpha(\coprod_k E_k)$. The $Q$-manifold versions of Theorem~\ref{iso-reg}
and Theorem~\ref{alf}  imply that the union of $p_k$
defines a completely regular map $p:X\to\alpha(\coprod B_k)$ with fibers $M$
where each $B_k$ is homeomorphic to $B$. By Proposition~\ref{universal}, $p=f^*(\nu_M)$ for some
continuous map $f:\alpha(\coprod_kB_k)\to\mathcal M(M)$.

We present $B=\lim_{\leftarrow}\{L_i,\phi^i_j\}$ as the limit of an inverse sequence of compact polyhedra.
Let $T$ be the natural compactification of $\coprod L_i$ by $B$. 
If $\mathcal M(M)$ is an ANE for compact metric spaces, then there is an extension 
$\bar f:W\to\mathcal M(M)$ to a neighborhood of $\alpha(\coprod B_k)$ in $\alpha(\coprod T_k)$.
We note that there is $k$ such that $T_k\subset W$ and the restriction of $\bar f$ to $T_k$ is null-homotopic.
Let $\bar p_k:Z_k\to T_k$ be the restriction of $\bar f^*(\nu_M)$ over $T_k$.

For sufficiently large $i$, there is an $\alpha_i\in h^*(L_i)$ that maps to $\alpha\in h^*(B)$. Let
$U_i \subset T$ be the compactification of  $L_i \coprod L_{i+1}\coprod\dots$ by $B.$ The bonding map
$B \to L_i$ factors through $U_i$; let $\alpha'_i\in h^*(U_i)$ be the image of $\alpha_i$. Let $V_i \subset Z_k$ be
the preimage in $Z_k$ of the copy of $U_i $ in $T_k$. Let $\beta'_i\in h^*(V_i)$ be the image of $\alpha'_i$.
Since $\alpha'_i$ 
maps to $\alpha$, the image of $\beta'_i$ in $h^*(E_k)$ is the same as the image of $\alpha$,
which is trivial by the assumption. Hence the image of $\beta'_i$ in $h^*(V_j)$ is trivial for
some $j > i$. If  $\alpha'_J\in h^*(U_j)$ is the image of $\alpha'_i$, then the image of $\alpha'_j$
in $h^*(V_j)$
is trivial. Finally, if $\alpha_j \in h^*(L_j)$ is the image of $\alpha_i$, then $\alpha_j$ maps to $\alpha'_j$
under
the map $U_j \to L_j$, which in turn goes to 0 under the restriction $V_j \to U_j$ of
$\bar p_k$. The composition $(\bar p_k)^{-1}
(L_j) \to V_j \to U_j \to L_j$ coincides with the restriction $(\bar p_k)^{-1}
(L_j) \to L_j$
of $\bar p_k$, and therefore $\alpha_j$ goes to zero under the latter restriction.
This contradicts the fact that $p_k$ is a trivial bundle over $L_j$.
\end{proof}

\end{document}